\documentclass[10pt]{amsart}
\usepackage[cp1251]{inputenc}
\usepackage[english,russian]{babel}
\usepackage{amsmath}
\usepackage{amssymb}
\usepackage{amsfonts}
\usepackage{t1enc,color,array,rotating,graphicx,mathtools}
\usepackage{soul}
\usepackage{latexsym,amssymb,amsmath,amsfonts,stmaryrd,algorithm,algpseudocode}
\usepackage[english]{babel}
\usepackage{hyperref}
\numberwithin{equation}{section}
\usepackage{graphicx,nicefrac,wrapfig}
\usepackage[textsize=footnotesize]{todonotes}

\usepackage{cleveref,multirow,tikz,xspace,comment,enumitem}

\newtheorem{lemma}{Lemma}
\newtheorem{theorem}{Theorem}
\newtheorem{definition}{Definition}
\newtheorem{corollary}{Corollary}
\newtheorem{proposition}{Proposition}
\newtheorem{remark}{Remark}
\newcommand{\R}{\mathbb{R}}
\newcommand{\K}{K}
\newcommand{\KR}{\widetilde{K}}

\newcommand{\Q}{\mathbb{Q}}
\newcommand{\im}{{\rm im}}
\newcommand{\dom}{{\rm dom}}
\newcommand{\Dyad}{{\rm Dyad}}
\newcommand{\Pt}{{\rm P}}
\newcommand{\SPt}{{\rm S}\Pt}

\newcommand{\AC}{{\rm AC}}
\newcommand{\Tot}{{\rm Tot}}
\newcommand \dotmm {\,{-}\hspace*{-6pt}\parbox{3pt}{\vspace*{-6pt}.}\hspace*{3pt}\,}
\newcommand{\dotminus}{\dotmm}
\newcommand{\BF}{{\rm BF}}
\newcommand{\Sp}{{\rm Sp}}

\begin{document}
\renewcommand{\refname}{References}
\renewcommand{\proofname}{Proof.}
\thispagestyle{empty}

\title[On the Computability of Ordered Fields]{On the Computability of Ordered Fields}

\author{{M.V.Korovina and O.V.Kudinov}}%
\address{Margarita Vladimirovna Korovina
\newline\hphantom{iii} A.P. Ershov Institute of Informatics Systems,
\newline\hphantom{iii} pr. Acad. Lavrentjev,6,
\newline\hphantom{iii} 630090, Novosibirsk, Russia}%
\email{rita.korovina@gmail.com}%
\address{Oleg Victorovich Kudinov
\newline\hphantom{iii} Sobolev Institute of Mathematics,
\newline\hphantom{iii} pr. Koptyug, 4,
\newline\hphantom{iii} 630090, Novosibirsk, Russia}%
\email{kud@math.nsc.ru}%

\thanks{\sc Korovina, M.V., Kudinov, O.V.,
On the Computability of Ordered Fields}
\thanks{\copyright \ 2020 Korovina M.V., Kudinov O.V}
\thanks{\rm O.V.Kudinov was  supported  by   the state contract
of the Sobolev Institute of Mathematics (project no. 0314-2019-0002)  and RFBR  project no. 20-01-00300a  and
 M.V.Korovina was    supported  by the state contract
of the IIS SBRAS (project no. 0317-2019-0003) and RFBR- JSPS
project no. 20-51-5000.}
\thanks{\it Received   2020, published    2020}%

\maketitle {\small
\begin{quote}
\noindent{\sc Abstract. }
In this paper we develop  general techniques  for structures of computable real numbers generated by  classes of total computable (recursive) functions with special restrictions on basic operations in order to investigate the following problems:
whether a generated structure is a real closed field and whether there exists a computable copy of  a generated structure.
We prove a series of theorems that lead to the result that there are no  computable copies
neither for polynomial time computable no  even for $\mathcal{E}_n$-computable real numbers, where $\mathcal{E}_n$ is a level  in  Grzegorczyk hierarchy, $n\geq 2$.
We also  propose a criterion of computable presentability of an archimedean ordered field.

 \end{quote}
}

\section{Introduction}

\noindent In the framework of computable model theory there have been investigated
 condi\-tions on the existence of computable copies
  for countable homogeneous boolean algebras \cite{Morozov}, for superatomic boolean algebras \cite{Gonch,Gonch1}, for ordered abelian groups \cite{Khisamiev} among others and established several negative  results for archimedean ordered fields \cite{Miller,KK_spectr}.

 \noindent
 Nevertheless,  till now there where no natural criteria on the existence of computable copies of ordered fields even in an  archimedean case. In this paper we fill this gap.

 \noindent We  are also going dipper to revile  relations between a class of computable (recursive) functions $\K$ and a structure $\KR$ of
 real $K$--numbers generated by $\K$.
 We propose natural restrictions on  a class $\K$ under which the structure $\KR$ is a real closed field.



\noindent  In this direction we investigate whether there exist  computable copies of   generated structures for
 popular classes of computable functions such as the polynomial time computable functions $\Pt$ and
Grzegorczyk classes  $\mathcal{E}_n,\, n\geq 2$.
  We establish that the corresponding real closed fields do not have computable copies and moreover the polynomial time computable real numbers as an abelian group does not have a computable copy as well.

  In order to do that we develop  techniques of index sets and multiple $m$-complete\-ness.
   On this way we have to   establish
  a criterion of $m$-completeness for tuples of c.e. sets and $\Sigma^0_2$-sets. From our point of view this criterion is
an interesting result   itself and can be used for different purposes.





%
The paper is organised a follows:

\noindent Section~\ref{prelim_sec}  contains preliminaries and  basic background.

\noindent  In Section~\ref{rcfield_sec}  we show under which natural restrictions on $\K$ the corresponding generated structure $\KR$ is a real closed field.

\noindent In Section~\ref{criterion_aof}  we  propose a criterion of the computable presentability of an archimedean ordered field.

 \noindent In Section~\ref{index_sets_constructivisiruemih_kr}  we define 3-tuple of index sets $(A_0,A_1,A_2)$ depending on $\K$ such that
$A_i\in\Sigma^0_2$  with the following  embedding property $A_0\subseteq A_1\subseteq A_2$.
In the Theorem~\ref{sigma02construct} we show that if the corresponding $\KR$ as a structure, in particular as  an abelian group, has a computable copy
then $A_0\cup (A_2\setminus A_1)\in \Sigma^0_2$.

 \noindent In Sections~\ref{m-universality}, \ref{index_sets} and \ref{without_constr} we develop techniques to  establish that   under natural assumptions on a class $K$   the 3-tuple $(A_0,A_1,A_2)$ is $m$-complete in the class of 3-tuples of $\Sigma^0_2$-sets with the embedding property.
It is well-known that in this case $A_0\cup (A_2\setminus A_1)\not\in \Sigma^0_2$ and therefore for the corresponding
$\KR$ there is no a computable copy.
It is worth noting than these classes contain polynomial time computable real numbers, computable real numbers generated by Grzegorczyk classes, in  particular $\mathcal{E}_2$ and beyond.


%

\section{Preliminaries}\label{prelim_sec}
We refer the reader to \cite{Rogers} and  \cite{Soarehbook} for  basic definitions and fundamental concepts of recursion theory \cite{Ko,Weihrauchbook} for computable analysis, \cite{ErshGonch} for computable model theory, \cite{Grzegorczyk} for
Grzegorczyk classes of computable (recursive) functions and \cite{ALR} for complexity classes.
We recall that, in particular, $\varphi_e$ denotes the partial  computable (recursive) function with  an index $e$ in the Kleene numbering. For simplicity of descriptions  we identify a function with its graph. We also use notations $W_e=\rm {dom}(\varphi_e)$, $ \overline{W_e}=\omega\setminus W_e$,  $\pi_e=\rm {im}(\varphi_e)$ and  for Cantor 3-tuple $(c,l,r)$ we associate  a number $n$ with the corresponding pair $<l(n),r(n)>$. We fix the set $\BF$ of standard  basic functions $\lambda x.0,\,s(x)$ and
 $I^n_m$, where $I^n_m(x_1,\dots,x_n)=x_m$ for $1\leq m\leq n$ and denote the total  computable numerical functions as $\mathbb{T}$ and $\Tot=\{n\mid \varphi_n\in \mathbb{T}\}$.
We fix a standard computable numbering $q:\omega\to \mathbb{Q}$  of the rational numbers and when it is clear from a context we use  the notation $q_n$ for $q(n)$. For tuples $z_1\dots,z_k$ of numbers or functions  we use the notation $\bar{z}$ when
 it is clear from a context.
 For the dyadic numbers we use the notation $Dyad=\{\frac{m}{2^{i+1}}\mid m\in \mathbb{Z},\,i\geq 0\}$,
  $\mathbb{Q}^{+}=\{q\in\mathbb{Q}\mid q>0\}$
  and $B(\alpha,r)=\{x\in\R\mid |x-\alpha|<r\}$ for a basic open ball with the center $\alpha\in\mathbb{Q}$ and the radius $r\in \mathbb{Q}^{+}$.

\subsection{Primitive Computable Reals}\label{KR}
Let $\K$ be a class of total computable numerical functions  with  the following restrictions:
it   contains the basic functions $\BF$ and $+,\,\cdot,$    closed under composition and
the following  bounded primitive  recursion scheme:
if $\alpha\,\,g,\psi\in K$ and $f$ is defined by

  \begin{eqnarray*}
f(\bar{x},y)=\left \{
\begin{array}{lll}
\alpha(\bar{x})&\mbox{ if }& y=0\\
\psi(\bar{x},y,f(\bar{x},y-1)) &\mbox{ if }& y\geq 1
\end{array}
\right.
\end{eqnarray*}
 and

\begin{align*}
f(\bar{x},y)\leq g(\bar{x},y)
\end{align*}
then $f\in K$.

 Then $\KR$ denotes the set of computable real numbers
\[
\{x\mid (\exists \Phi\in\K)(\forall n>0)|x-q(\Phi(n))|\leq\frac{1}{2^n}
\}.
\]
 This is equivalent to the existence of $F\in\K$ s.th.
$(\forall n \geq 0)(\forall N>n)\,|a_N-a_n|\leq\frac{1}{2^n}\wedge \lim_{n\to\infty} a_n=x$, where
$a_n=q(F(n))$.
 If $\K$ contains  $\lambda x.2^x$ then an element of $\KR$ is called  a $\K$--number.
In particular,  if $\K$ is the class of all primitive computable functions then $\KR$ is called the  primitive computable real numbers. The same for $\mathcal{E}_n$, $n\geq 3$. For $\mathcal{E}_2$ and polynomial time computable reals see Section~\ref{index_sets_constructivisiruemih_kr}.
\begin{remark}
In Section~\ref{index_sets_constructivisiruemih_kr} we are going to consider not so rich classes as above. Therefore the~definition of $\KR$ will be modified. In particular, in Section~\ref{index_sets_constructivisiruemih_kr} and further the notion $x\in\KR$ will differ from  the notion $\K$-number.
\end{remark}

\subsection{ Computable Presentations}

 We say that a structure ${\mathcal A}=\left <A,\sigma\right >$ admits an  computable presentation  (copy) if there is a numbering $\nu:\omega\to A$ such that  the relations and operations  from $\sigma$ including equality
 are computable  with respect to the numbering $\nu$. The pair $({\mathcal A},\nu)$ is called a computable structure and the numbering $\nu$ is called its computable presentation (constructivisation).
If only operations are computable with respect to the numbering $\nu$, a structure  $({\mathcal A},\nu)$  is called a~numbered (effective) algebra.
\section{When $\KR$ is a real closed field}\label{rcfield_sec}
Let us fix $\K$ with the restrictions from Section~\ref{KR}.
\begin{proposition}
The corresponding structure $\KR=(\KR,+,\cdot,\leq)$ is a real closed field.
\end{proposition}
\begin{proof}
The claim that $\KR$ is closed under addition, subtraction,
multiplication and division is straightforward.
To complete the proof, we  show that the roots of polynomials
with coefficients in $\KR$ are also in the class $\KR$.
Assume contrary there exists a polynomial  $p(x)= \sum^n_{i=0}a_i x_i\in\KR[x]$ of minimal degree which has a root $x$ in $\R$ but not in $\KR$. The polynomial $p$ does not have multiple roots since in opposite case it is possible to compute $g={\rm G.C.D.}(p(x),p^\prime(x))$ provided by exact knowledge of zero coefficients
 of $q$ with ${\rm deg}(q)<{\rm deg}(p)$.
So the coefficients $\bar{a}=(a_0,\dots,a_n)$ of $p$ satisfy the following formula:
\[
\psi(\bar{a})=\psi_1(\bar{a})\vee \psi_2(\bar{a}),
\]
where
\[
\psi_1(\bar{a})\leftrightharpoons
(\exists A)(\exists B)(\exists\epsilon>0)
\Big(
A<B\wedge p(A)<-\epsilon\wedge\, p(B)>\epsilon\wedge (\forall x\in [A,B])\, p^\prime(x)>0
\Big)
\]
and
\[
\psi_2(\bar{a})\leftrightharpoons
(\exists A)(\exists B)(\exists\epsilon>0)
\Big(
A<B\wedge p(A)>\epsilon\wedge\, p(B)<-\epsilon\wedge (\forall x\in [A,B])\, p^\prime(x)<0
\Big).
\]

 W.l.o.g. we assume $\R\models\psi_1(\bar{a})$.
 By the Uniformity Principal \cite{KK_up}, the formula $\psi_1$  can be effectively transformed to a formula
 \[
 \bigvee_{A,\, B\in\Q, A<B}\,  \bigvee_{\epsilon\in\Q^+}\Theta_{A,\, B,\,\epsilon}(\bar{a}),
 \]

\noindent where $\Theta_{A,\, B,\,\epsilon}(\bar{x})$ is a uniformly computable disjunctions of $\exists$-formulas without equality.
 Therefore $\psi_1$ defines an effectively enumerable subset of $\R^{n+1}$.
 Suppose $\R\models \Theta_{A,\, B,\,\epsilon}(\bar{a})$. For simplicity of the further reasoning we fix the product of open balls containing $\bar{a}$:
 \[
 \prod_{i=0}^n B(\alpha_i,r_i)\subseteq\{\bar{x}\mid\R\models\Theta_{A,\, B,\,\epsilon}(\bar{x})\}
 \]

 \noindent and the corresponding  $\alpha_i\in\Q$ and $r_i\in\Q^+$, $0\leq i\leq n$.
Since $a_i\in\KR$ for $0\leq i\leq n$, in the framework of $\K$ one can effectively find  rational tuple
$b_0,\dots,b_n$ such that for all $0\leq i\leq n$,
\begin{itemize}
\item  $b_i\in B(\alpha_i,r_i)$ and
\item  $|b_i-a_i|\leq\frac{1}{2^{3m}}$,
\end{itemize}
where $m$  is an argument  of this computations.
Let $M\in\Q^+$ be a bound on $a_i$ and $B$, i.e., $|a_i|<M$ for $0\leq i\leq n$ and $|B|<M$.
Having $m$ and  the required precision $\frac{1}{2^m}$, in the framework of $\K$ one can effectively find $y\in\Q\cap [A,B]$ such that
\[
\mid \sum_{i=0}^n b_iy^i\mid<\frac{1}{3m}
\]
One can assume that $m$ is sufficiently big , i.e.,
$2^m> M^n+\dots+1$. It is worth noting that
$2^m\cdot\epsilon>1$ and
$\mid p(y)\mid\leq \frac{1}{2^{3m}}\cdot(1+\dots|y|^n)\leq \frac{1}{2^{3m}}\cdot 2^m=\frac{1}{2^{3m}}$.
By the mean value  theorem,  for all $x,\, y\in [A,B]$ there exists $c\in [A,B]$
such that $p(x)-p(y)=(x-y)\cdot p'(c)$.
If $x$ is the root of $p$ in the interval $[A,B]$ then
\[
|x-y|\leq \frac{|p(y)|}{\epsilon}\leq\frac{1}{2^{2m}}\cdot2^m\leq\frac{1}{2^m}.
\]

So $y\in \mathbb{Q}$ is an approximation of  the root $x$ with the precision $\frac{1}{2^m}$.
Therefore, $x\in\KR$, a contradiction.
\end{proof}
An discussion  that the previous proposition is an refinement of the result in \cite{Ko} for $\Pt$-numbers
one can find in Section~\ref{index_sets_constructivisiruemih_kr}.

\section{Criterion of Computable Presentability and Archimedean Part}\label{criterion_aof}

Let $L=(L,\leq)$ be linearly ordered and $\mathbb{Q}\subseteq L$.  Assume $\mu:\omega\to L$ is a numbering.

With $L$ we associate 2 families:
\begin{align*}
&A_k=\{ n\mid q_n\leq \mu(k)\}\\
&B_k=\{ n\mid q_n\geq \mu(k)\}
\end{align*}
and naturally define $S_k=A_k\oplus B_k=\{2n\mid n\in A_k\}\cup \{2n+1\mid n\in B_k\}$ and $S_L=\{S_k\mid k\in\omega\}$.
\begin{remark}\label{rm_3}
 It is worth noting that $S_L$ does not depend on the choice of $\mu$.  By the way, it is easy to see that if $\mu$ is  a computable presentation of an ordered field
 $F=(F,+,\cdot,\leq)$  then $\mu\geq q$, i.e., $q_n=\mu(h(n))$
for a computable function $h:\omega\to\omega$ and  the family $S_F$ is computable.
\end{remark}
The following theorem provides  a criterion of computable presentability of  an archimedean ordered field.

\begin{theorem}\label{criterion_comp_copy}
Let $F=(F,+,\cdot,\leq)$ be an archimedean ordered field, $\mu:\omega\to F$ be its numbering such that
$(F,\mu)$ is a numbered  algebra.
Then the family $S_F$ is computable if and only if $(F,\mu)$ is a computable copy.
\end{theorem}
\begin{proof}

\noindent The claim  $\rightarrow$ follows from Remark~\ref{rm_3}.

\noindent  For the claim  $\leftarrow$ we assume that $S_F$ is computable.
Let $0=q_i$.  Since for some $a\in\omega$ we have $-1=\mu(a)$, the substraction is defined as  $\mu(n)-\mu(m)=\mu(n)+\mu(a)\cdot\mu(m)$. So, it is clear that $\mu(n)=\mu(m)$ iff $\mu(n)-\mu(m)=\mu(k)\wedge 2i\in S_k\,\wedge\, 2i+1\in S_k $.
Therefore equality is computably enumerable. Since positive fields are computable, equality is computable.

 It is easy to see that  order is also computable. Indeed,
 \begin{align*}
 &\mu(n)<\mu(m)\mbox{ iff }
(\exists k)(\exists l)\, \mu(n)<q_k<q_l\leq\mu(m),\\
&\mu(n)\leq \mu(m)\mbox{ iff } \mu(n)<\mu(m)\vee \mu(n)=\mu(m),\\
&\mu(n)\not\leq \mu(m)\mbox{ iff }\mu(m)<\mu(n).
 \end{align*}
\end{proof}

Now we show that the requirement that
$(F,\mu)$ is a numbered  algebra one can not avoid to establish that $F$ has a computable copy.
For that we start with  general definitions and observations on the archimedean part of an ordered field.
\begin{definition} For $x\in F$, where $F$ is an ordered field we define
$\Sp(x):F\to \mathbb{R}\cup\{-\infty,+\infty\}$ as follows:
  \begin{eqnarray*}
\Sp(x)=\left \{
\begin{array}{lll}
+\infty&\mbox{ if }& x>\mathbb{Q}\\
y \in\mathbb{R} &\mbox{ if }& |x-y|< \mathbb{Q}^{+}\\
-\infty&\mbox{ if }& x<\mathbb{Q}.
\end{array}
\right.
\end{eqnarray*}
\end{definition}

It is worth noting that $y$ is uniquely defined and  if $\Sp(x)\not\in\{-\infty,+\infty\}$ we say about a finite spectrum  of
$x$.

It is worth noting that Zorn lemma provides the fact that any ordered field   $(F, +, \dot, 0, 1, \leq)$  has some maximal archimedean subfield $F_0 \leq F$. The main issue of the following proposition is that all  maximal archimedean subfield are isomorphic each other.
\begin{proposition}\label{archim_part}
Let $(F,+,\cdot,0,1,\leq)$ be an ordered real closed field. Assume $F_0$ is its~maximal archimedean subfield.
Then  $F_0$ is a real closed field  and the greatest archimedean subfield with respect the following pre-order on subfields of $F$:
\begin{align*}
F_1 \leq F_2 \mbox{ iff there is an isomorphic embedding $\varphi: F_1\to F_2$ as ordered fields}.
\end{align*}

Actually,
$F_0\simeq \Sp(F_0)=\Sp(F)\setminus\{-\infty,+\infty\}$ considered as subfields of the reals and
for any archimedean $F_1 \leq F$ it holds $F_1 \cong \Sp(F_1) \leq \Sp(F_0)$.
In particular, all maximal archimedean subfields of $F$ are isomorphic each other.
\end{proposition}
\begin{lemma}\label{rcf}
$F_0$ is a real closed field.
\end{lemma}
\begin{proof} (Lemma~\ref{rcf})
We are going to check that monic  (unitary) odd polynomials and $x^2-b$, $(b>0)$ from $F_0[x]$ have roots in $F_0$.
For the case $x^2-b$, $(b>0)$ the claim is straightforward. Indeed, since $F$ is real closed there exists $x_0>0$ such that
$x_0^2=b$ and $x_0>0$. If $x_0\not \in F_0$ then $F_0(x_0)$ is non-archimedean, so, for some $a\in F_0$, $|x_0-a|<\mathbb{Q}^{+}$.
Then $|x_0-a|<\mathbb{Q}^{+}$ and $|a^2-b|<\mathbb{Q}^{+}$, a contradiction.

Assume $p(x)$ is a monic  polynomial of the least odd degree that does not have roots in $F_0$. Fix $x_0\in F$ such that $p(x_0)=0$, so
$\Sp(x_0)$ is finite.
Moreover, it follows   that $p(x)$ is irreducible over $F_0$ since the equality $p=p_1\cdot p_2$, where $p_1$, $p_2$ are monic polynomial of degree
greater than $0$ leads to a contradiction to the choice of $p$.

If $x_0\not \in F_0$ then $F_0(x_0)$ is non-archimedean, so,  for some $q\in F_0[x]$, ${\rm deg}(q)< {\rm deg}(p)$ we have
$|q(x_0)|<\mathbb{Q}^{+}$. Since $p$ and $q$ do not have a common factor, ${\rm Res}(p,q)\in F_0\setminus\{0\}$.
It is well-known that ${\rm Res}(p,q)=p(x)A(x)+q(x)B(x)$ for some $A,\, B\in F_0[x]$. Therefore, $|{\rm Res}(p,q)|< \mathbb{Q}^{+}$,
a contradiction.
\end{proof}

 To finish the proof of the proposition first let us note that $\Sp\upharpoonright F_0$ is an isomorphic embedding.  It is sufficient to show that
$\Sp(F)\setminus \{-\infty,+\infty\}=\Sp(F_0)$.
Let $x_0\in F$ such that $\Sp(x_0)$ is finite and $x_0\not\in F_0$. Since $F_0(x_0)$ is non-archimedean, for some $a,\, b\in F_0[x]$ we have $|\frac{a(x_0)}{b(x_0)}|<\mathbb{Q}^{+}$. W.l.o.g. assume that the fraction is irreducible. We have two cases:

1) $|b(x_0)|>\mathbb{Q}$. This contradicts the finiteness of $\Sp(x_0)$.

2) $|a(x_0)|<\mathbb{Q}^{+}$. Then $\Sp(a)(\Sp(x_0))=0$. That means that $\Sp(x_0)$ is algebraic over the field $\Sp(F_0)$.
Since, by Lemma~\ref{rcf}, $F_0$ and $\Sp(F_0)$ are real closed we have $\Sp(x_0)\in\Sp(F_0)$,
that finishes the proof of Proposition~\ref{archim_part}.

Now  $F_0$  is called  an archimedean part of $F$.

In \cite{Miller} R. Miller and  V. O. Gonzales  constructed  a computable ordered field $F^\prime$ such that no
maximal archimedean subfields of   $F^\prime$  have  computable copies.
In our terms it means precisely that the archimedean part of   $F^\prime$ does not gave a computable copy.
We take this example to illustrate that the requirement that
$(F,\mu)$ is a numbered  algebra one can not avoid to establish that $F$ has a computable copy.
This note highlights the importance of the  results  in \cite{Miller}.

Assume now that $F^\prime$ is computable and   its archimedean part $F=F^\prime_0$ does not have a computable copy.
By Theorem~\ref{criterion_comp_copy} the family $S_{F^\prime}$ is computable.
While  $S_{F^\prime}\neq S_F$, the following formula defines $S_F$:
\begin{align*}
A\in S_F\leftrightarrow (\exists q_i)(\exists q_j)\big( A\in S_{F^\prime}\wedge q_i<q_j\wedge 2i\in A \wedge 2j+1\in A\big).
\end{align*}
Therefore $S_F$ is computable, however $F$ does not have a computable copy.

\section{  Index sets vs. Computable Presentability}\label{index_sets_constructivisiruemih_kr}
In this section we assume
 $\K$ is a class of total computable numerical functions. We associate with $\K$ the class
 $\K^{\{0,1,2\}}=\{f\in \K\mid{\rm im}(f)\subseteq \{0,1,2\}\}$.

  By $\AC$ we denote the almost  constant  functions, i.e.,
  $\AC=\{f:\omega\to\omega\mid (\exists c\in \omega)(\exists x\in \omega)(\forall y>x) f(y)= c\,\} $
and use the following notation $f=^\ast c$ for a function $f \in \AC$ with  the evidence $c \in \omega$.
We assume that $\AC\subseteq \K$.

 We proceed with the definition of  the corresponding structure $\KR$.
It is worth noting that for classes with the restriction from   Section~\ref{KR} the previous and the following definitions are equivalent.
For $f\in\K$ let us denote
\[
 \bar{f}=\sum_{i=0}^\infty\frac{f(i)-1}{2^{i+1}}.
\]

Then we define  $\widetilde{\KR}=\{\bar{f}\mid f\in \K^{\{0,1,2\}}\}$ and $\KR=\{m+\bar{f}\mid m\in\mathbb{Z},\, \bar{f}\in \widetilde{\KR}\}$.

It is easy to see that $\widetilde{\KR}$ contains $\Dyad\cap[-1,1]$ since $AC\subseteq K$
and $1=\sum_{i=0}^\infty\frac{2-1}{2^{i+1}}\in \widetilde{\KR}$.

Motivations why we from right now do not identify $\K$-numbers and elements of $\KR$ are as follows.

Let us consider the following example of the $\Pt$--numbers computed in polynomial time, where a polynomial is applied to the length of an argument. These numbers form a real closed field \cite{Ko} and denoted there as $\Pt_{CF}$. It turns out that for $x\in[-1,1]$ nobody can state that
\[x \mbox{ is $\Pt$-number iff }(\exists f\in\Pt^{\{0,1,2\}})\, x=\bar{f}. \]
Actually, in these terms one can define  a new class $\K_\Pt$ containing functions that are computed in polynomial time, where a polynomial is applied to an argument instead of its length.

It is easy to see that the class $\K_\Pt$ is closed under the bounded primitive recursion scheme, in particular, Proposition~\ref{rcfield_sec}
is applicable to $\widetilde{\K_\Pt}$.
For this class we can state that, for $x\in[-1,1]$,
\[x \mbox{ is $\Pt$-number iff }(\exists f\in\K_\Pt^{\{0,1,2\}})\, x=\bar{f}. \]
That means that  the $\Pt$-numbers are exactly $\widetilde{\K_{\Pt}}$ (c.f. \cite{Ko}).
Further it makes sense to analyse  the set $\widetilde{\Pt}$ that consists from elements $x$ that the computation of i-th sign of $x$ after the comma requires not more than  $N$ steps, where   $N$ is a polynomial  on the length of $i$.
Unfortunately we can not state that $\widetilde{\Pt}$ is an subgroup of $(\mathbb{R},+)$ since this question is still an open problem.
This leads us to introduce  a new definition of  a structure, see below.
 For right now it is worth noting that the $\SPt$-numbers (i.e., $\widetilde{\Pt}$) is naturally associated with the class $\SPt$ of  computable functions which computations require   no more steps than $[\log_2([\log_2(x_i)])]$, $i\leq s$, where $\bar{x}=(x_1,\dots,x_s)$ is an $s$-tuple of arguments. The corresponding structure looks like small however in Section~\ref{without_constr} we show that neither the $\Pt$-numbers no   $\SPt$-numbers  have computable copies.


Our next example is related to an appropriate definition of $\mathcal{E}_2$-numbers. It is well-known  that
$\lambda x.2^x\not\in \mathcal{E}_2$ and $\mathcal{E}_2={\rm LogSpace}$, i.e.,
$f\in\mathcal{E}_2$ if and only if there exists $c\in\omega $ such that for all $x_1,\dots,x_m$  a computation of $f(\bar{x})$ requires not more that $c\cdot {\max}_{i\leq m}|x_i|$ cells of a tape (memory), where $|x|$ denotes the length of $x$ and w.l.o.g
it can be substituted by $[\log_2(x)]$. According to the approach in \cite{Ko} it is natural to have the following definition:

for $x\in[-1,1]$,
\begin{align*}
x \mbox{ is a  $\mathcal{E}_2$-number} \leftrightarrow (\exists f\in{\rm DSpace}^{\{0,1,2\}}(n))
\,\, x=\bar{f},
\end{align*}

and for $x\in\mathbb{R}$,
\begin{align*}
 x \mbox{ is a  $\mathcal{E}_2$-number} \leftrightarrow
 (\exists m\in\mathbb{Z})(\exists  y\in \widetilde{\widetilde{{\rm DSpace}(n)}})\,\,\, x=m+y.
\end{align*}
Therefore $x$ is a $\mathcal{E}_2$-number iff $x\in {\widetilde{{\rm DSpace}(n)}}$.

\noindent According the observations  above  we introduce the definition of a structure on $\KR$.
\begin{definition}\label{str}
Suppose $\K$ is a class of total computable numerical functions, $\KR$ is the set of reals generated by $K$.
Then we associate with $\KR$  a structure  $\KR=(\KR,0,Q^3_{+},Q^3_{-})$, where
\begin{align*}
&\KR\models Q^3_{+}(x,y,z)\leftrightarrow x+y\leq z\\
&\KR\models Q^3_{-}(x,y,z)\leftrightarrow x+y\geq z.
\end{align*}
\end{definition}
It is easy to see that, since the graph of addition is computable, if  a structure $\KR$ has a computable copy $(\KR,\mu)$ then the sets $\mu^{-1}(\Dyad)$ and $\mu^{-1}(\mathbb{Q})$ are computably enumerable.
For example,
\begin{align*}
 \mu(n)\in \Dyad \leftrightarrow (\exists k\in\omega)(\exists l\in\mathbb{Z})\, 2^k\cdot \mu(n)+l=0.
\end{align*}
To proceed    further we define index sets
\begin{align*}
&A_0=\{n\mid \pi_n\subseteq \{0,1,2\}\wedge n\not\in \Tot\},\\
&A_1=\{n\mid \pi_n\subseteq \{0,1,2\}\wedge ( n\not\in \Tot\vee \varphi_n\in\AC)\},\\
&A_2=\{n\mid  \pi_n\subseteq \{0,1,2\}\wedge (n\not\in \Tot\vee\varphi_n\in\K^\ast)\}=\\
&\{n\mid  \pi_n\subseteq \{0,1,2\}\wedge (n\not\in \Tot\vee\overline{\varphi_n}\in\KR)\}
\end{align*}
where  $\K^\ast=\{\varphi_n\mid \overline{\varphi_n}\in \KR\cap [-1,1] \}$.
\begin{theorem}\label{sigma02construct}
Suppose $\K$ is a class of total computable numerical functions. If  the structurer $\KR$ generated by $K$ has a computable presentation then
$A_0\cup (A_2\setminus A_1)\in \Sigma^0_2$.
\end{theorem}
\begin{proof}

Let $\mu:\omega\to\KR$ be a computable presentation.
 Since the set $E=\{n\mid -1\leq\mu(n)\leq 1\}$ is computable there exists a computable function $h$  such that
 ${\rm im}(h)=E$ and $\tilde{\mu}=\mu\circ h$ is  a computable numbering of $\KR\cap [-1,1]$.
Assume $x=\tilde{\mu}(n)$. Now we construct a map $\nu:\omega\to \mathbb{T}$ by induction:
\begin{align*}
& \nu(n)(0)=1\\
& \nu(n)(s+1)=\left \{
\begin{array}{lll}
0 & \mbox{ if }&  x<x_s\\
1 &\mbox{ if }&  x=x_s\\
2 &\mbox{ if }&  x>x_s,
\end{array}
\right.
\end{align*}
 where $x_s=\sum_{i\leq s}\frac{\nu(n)(i)-1}{2^i}$.
 Since $|x_s-x|\leq \frac{1}{2^s}$, $|x_{s+1}-x|\leq \frac{1}{2^{s+1}}$. From $\AC\subseteq K$ it follows that
 $x_s\in\KR$.

 We have the following properties:
 $\nu(n)$ is total,  $\nu(n)\in \mathbb{T}$ and $\nu(n)$ provides  a sign-digit representation of $x$.

From properties of $\mu$ it follows that $\nu(n)\in\AC$
  is a $\Sigma^0_1$--condition since $\nu(n)\in\AC\leftrightarrow \mu(n)\in\Dyad$ and
$\mu(n)\in \Dyad\leftrightarrow (\exists k\in\omega)(\exists l \in\mathbb{Q})\, 2^k\cdot\mu(n)+l=0$.
As  a corollary,  $Y=\{n\mid\nu(n)\in AC\}$ is computably enumerable.
Now we show that $A_0\cup A_2\setminus A_1\in \Sigma^0_2$. Let us note that
\begin{align*}
& n \in A_0\cup (A_2\setminus A_1)\leftrightarrow\\
&
n\in A_0\vee
\Big(n\in A_2\wedge \Big( (\exists m\in \omega\setminus Y)\, \overline{\nu(m)}=\overline{\varphi_n}
\vee n\not\in\Tot\Big)\Big).
\end{align*}

We have the following:
\begin{itemize}
\item The relation $n\not\in Y$ is $\Pi^0_1$.

\item The relation $\overline{\nu(m)}=\overline{\varphi_n}$ is $\Pi^0_1$.
It follows from the following observations.
Let

\begin{align*}
\Phi(f,g)\leftrightarrow \leftrightharpoons
(\exists s>0)\,
  \Big|\sum^{s}_{k=0}\frac{f(k)-1}{2^{k+1}}- \sum^{s}_{k=0}\frac{g(k)-1}{2^{k+1}}\big|>\frac{1}{2^{s-1}}
.
\end{align*}
Then for $f,\, g\in \mathbb{T}^{\{0,1,2\}}$, $\overline{f}\neq \overline{g}\leftrightarrow \Phi(f,g)$. So, $\overline{\nu(m)}=\overline{\varphi_n}\leftrightarrow \neg\Phi(\nu(m),\varphi_n)$.

 \item The relation $n\in A_2$ is $\Sigma^0_2$ since
\begin{align*}
n\in A_2\leftrightarrow n\not\in\Tot\vee (\exists m\in\omega)\neg \Phi(\nu(m),\varphi_n)
.
\end{align*}
\item The relation $n\not\in Tot$ is $\Sigma^0_2$.
\end{itemize}

Therefore $A_0\cup (A_2\setminus A_1)\in \Sigma^0_2$.
\end{proof}
It is worth noting that in many cases $\KR$ is an abelian ordered group, in particular $\widetilde{\mathcal{E}}_n$, $n\geq 2$,
$\Pt_{CF}$-numbers.
The same proof is valid when one consider just a computable presentation $\mu$ of  a linear ordered $(\KR,\leq)$ with the requirement that $\mu\geq d$, where $d$ is a standard computable presentation of $(\Dyad,\leq)$.

\section{Criterion of m-completeness for  tuples of $\Sigma^0_1$ and $\Sigma^0_2$ sets }\label{m-universality}
In this section for $s\geq 1$ we consider  $s$-tuples    $(A_0,\dots,A_{s-1})$, where all $A_i$ are either $\Sigma^0_1$-sets or
all $A_i$ are $\Sigma^0_2$--sets.

For uniformity of  a presentation we introduce a symbol $l$ where   $l\in\{1,2\}$, a  relation $\sim_l$ on sets and an oracle $z_l$ that have the following interpretation.
If $l=1$ then $ A\sim_l B$ means that  $A$ and $B$ are equal and the oracle $z_l=\emptyset$. If $l=2$ then $ A\sim_l B$ means that $(A\setminus B)\cup (B\setminus A)$ is finite, i.e.,  $A$ and $B$ are almost equal, denoted $A=^* B$. The oracle $z_l=Kw$,
where $Kw=\{n\mid \varphi_n(n)\downarrow\}$ or could be any creative set.
We generalise ideas of the criterion of $m$-completeness of $\Sigma^0_1$-sets and $\Sigma^0_2$-sets in
\cite{Arslanov_0} to fit  $m$-completeness of  $s$-tuples of $\Sigma^0_1$-sets and $\Sigma^0_2$-sets that requires modifications of  concepts and definitions.

\begin{remark}\label{rmalmosteq}
It is well known ( see c.f. \cite{Arslanov})  that on the set of all computable enumerable sets of $\omega$, given any (partial) $\Sigma^0_2$-function $f$
one can effectively construct  a total computable function $F$ such that
$(\forall x\in \dom(f))\, W_{f(x)}=^\ast W_{F(x)}$, moreover $\varphi_{f(x)}=^\ast \varphi_{F(x)}$.
\end{remark}

\begin{definition}
Let  $(F_0,\dots,F_{s-1})$ be an $s$-tuple of  functions, where $F_i:\omega^s\to\omega$, $0\leq i\leq s-1$
and $(A_0,\dots,A_{s-1})$ be an $s$-tuple of $\Sigma^0_l$-sets.
We say that $(F_0,\dots,F_{s-1})$ is m-reducible to $(A_0,\dots,A_{s-1})$, denoted as  $(F_0,\dots,F_{s-1})\leq_m (A_0,\dots,A_{s-1})$,
if there exist computable functions $h:\omega^s\to\omega$, $a_i:\omega^s\to\omega$, $b_i:\omega^s\to\omega$,  $0\leq i\leq s-1$,   such that

 \begin{eqnarray*}
F_i(\bar{x})=\left \{
\begin{array}{lll}
a_i(\bar{x}) & \mbox{ if }&  h(\bar{x})\in A_i\\
b_i(\bar{x}) &\mbox{ if }&  h(\bar{x})\not\in A_i.
\end{array}
\right.
\end{eqnarray*}

\end{definition}

It is easy to see that this definition is a  generalisation of the common m-reducibility of c.e. sets \cite{Rogers}.
\begin{lemma}\label{lem_m_reduc}
 For  s-tuples $X$ and $A$ of $\Sigma^0_l$-sets if $(F_0,\dots,F_{s-1})\leq_m (X_0,\dots,X_{s-1})$ and $(X_0,\dots,X_{s-1})\leq_m (A_0,\dots,A_{s-1})$ then
$(F_0,\dots,F_{s-1})\leq_m (A_0,\dots,A_{s-1})$.
\end{lemma}
\begin{proposition}\label{criteria_m_compl}
Let  $(A_0,\dots,A_{s-1})$ be a $s$-tuple of $\Sigma^0_l$-sets.
The following claims are equivalent.
\begin{enumerate}
\item $(A_0,\dots,A_{s-1})$ is m-complete in the class of  $s$-tuples of $\Sigma^0_l$-set.
\item  There exists  a computable function (its productive function) $H:\omega^s\to\omega$, i.e., for all $x_0,\dots\,x_{s-1}\in\omega$
\[
H(x_0,\dots\,x_{s-1})\in \bigcap_{i=0}^{s-1}
\Big( (A_i\cap W_{x_i})\cup (\overline{A_i}\cap \overline{W}_{x_i})
\Big)
\]
\item There exists a $s$-tuple of  functions $(F_0,\dots,F_{s-1})$, where $F_i:\omega^s\to\omega$, $0\leq i<s$, such that
   \begin{enumerate}
   \item $(F_0,\dots,F_{s-1})\leq_m (A_0,\dots,A_{s-1})$,
   \item $W_{F_i(\bar{x})}\not\sim_l W_{x_i}$ for all $0\leq i<s$.
   \end{enumerate}
   \item There exists a $s$-tuple of  functions $(F_0,\dots,F_{s-1})$, where $F_i:\omega^s\to\omega$, $0\leq i<s$, such that
   \begin{enumerate}
   \item $(F_0,\dots,F_{s-1})\leq_m (A_0,\dots,A_{s-1})$,
   \item $\varphi_{F_i(\bar{x})}\not\sim_l \varphi_{x_i}$ for all $0\leq i<s$.
   \end{enumerate}
\end{enumerate}
\end{proposition}
\begin{proof}
\noindent $1)\leftrightarrow 2)$. For $l=1$ the equivalents of  the statements can be found in \cite{Ersh_TH}.
The existence  of $m$-complete $s$-tuple of computably enumerable  sets has been also established there.
For $l=2$ we only need  a relativisation to the oracle $z_2$ which also could be found in \cite{Ersh_TH}.

\noindent $1)\to3)$.
Without loss of generality we assume $s=3$.

\noindent \textbf{Case 1:  l=1.}
First we
take the following m-complete 3-tuple:
 \begin{align*}
& X_0=\{n\mid \varphi_n(0)\downarrow\},\\
 &X_1=\{n\mid \varphi_n(1)\downarrow\},\\
 &X_2=\{n\mid \varphi_n(2)\downarrow\}.\\
  \end{align*}
Further on by By Lemma~\ref{lem_m_reduc}  the considerations below will hold for any m-complete 3-tuple.
  By Graph theorem \cite{Rogers} we construct a computable sequence $B=\{B_{x_0x_1x_2}\}_{x_0x_1x_2\in\omega}$ of partial  computable functions such that
   \begin{align*}
  & B_{x_0x_1x_2}(i)=\varphi_{x_i}(i) \mbox{ if } i=0,1,2\\
  & B_{x_0x_1x_2}(k)\uparrow \mbox{ if } k>2.
     \end{align*}
 Then there exists  a computable function $h:\omega^3\to\omega$ such that     $B_{x_0x_1x_2}=\varphi_{h(x_0,x_1,x_2)}$.
 By definition of $B$ we have    $h(x_0,x_1,x_2)\in X_i\leftrightarrow x_i\in X_i$ for $i\leq 2$.
  To finish  the construction of $F_i$ we take $a_i$ and $b_i $   for $i\leq 2$ as follows.

 \begin{align*}
& a_0 \mbox{ is an index of } \bot; \\
& b_0 \mbox{ is an index of  the function } \{<0,0>\}; \\
& a_1 \mbox{ is an index of } \bot; \\
& b_1 \mbox{ is an index of the function}  \{<1,0>\}; \\
& a_2 \mbox{ is an index of } \bot; \\
&  b_2 \mbox{ is an index of the function}  \{<2,0>\};
 \end{align*}

 We define  for $i\leq 2$
  \begin{eqnarray*}
F_i(\bar{x})=\left \{
\begin{array}{lll}
a_i & \mbox{ if }&  h(\bar{x})\in X_i\\
b_i &\mbox{ if }&  h(\bar{x})\not\in X_i.
\end{array}
\right.
\end{eqnarray*}
 By construction $(F_0,F_1,F_2)$ is m-reducible to $(X_0,X_1,X_{2})$. Let us show that
  $W_{F_i(\bar{x})}\neq W_{x_i}$.    Fix $i$. Assume $h(\bar{x})\in X_i$. Since $F_i(\bar{x})=a_i$, $W_{F_i(\bar{x})}=\emptyset$.
  At the same time $W_{x_i}\neq \emptyset$ since $x_i\in X_i$ and $\varphi_{x_i}(i)\downarrow$.
  Assume $h(\bar{x})\not\in X_i$. Since $F_i(\bar{x})=b_i$, $\varphi_{b_i}(i)\downarrow$  and $i\in W_{F_i(\bar{x})}\neq\emptyset$.
  At the same time $\varphi_{x_i}(i)\uparrow$, i.e., $i\not\in W_{x_i}$.

  Therefore $(F_0,F_1,F_2)$ is a required $3$-tuple.

\noindent \textbf{Case 2: l=2.}
First we
take the following  3-tuple:
 \begin{align*}
& Z_0=\{n\mid W_n\cap 3\omega \mbox{ is finite}\},\\
 &Z_1=\{n\mid W_n\cap (3\omega+1) \mbox{ is finite}\},\\
 &Z_2=\{n\mid W_n\cap (3\omega+2) \mbox{ is finite}\}.\\
  \end{align*}
By analogy to the case 1 we chose $h:\omega^3\to\omega$ such that for all $\bar{x}=(x_0,x_1,x_2)$ and every $k\in\omega$
 we have $\varphi_{h(\bar{x})}(3k+1)=\varphi_{x_i}(3k+1)$ for $i\leq 2$.
  To finish  the construction of $F_i$ we take $a_i$ and $b_i $   for $i\leq 2$ as follows.

 \begin{align*}
& a_0 \mbox{ is an index of constant zero function}; \\
& b_0 \mbox{ is an index of } \bot; \\
& a_1 \mbox{ is an index of constant zero function}; \\
& b_1 \mbox{ is an index of } \bot; \\
& a_2 \mbox{ is an index of constant zero function}; \\
&  b_2 \mbox{ is an index of } \bot;
 \end{align*}

 We define  for $i\leq 2$
  \begin{eqnarray*}
F_i(\bar{x})=\left \{
\begin{array}{lll}
a_i & \mbox{ if }&  h(\bar{x})\in Z_i\\
b_i &\mbox{ if }&  h(\bar{x})\not\in Z_i.
\end{array}
\right.
\end{eqnarray*}
 By construction $(F_0,F_1,F_2)$ is m-reducible to $(Z_0,Z_1,Z_{2})$. By analogy to the case~1 we have
  $W_{F_i(\bar{x})}\neq^\ast W_{x_i}$. Therefore $(F_0,F_1,F_2)$ is a required $3$-tuple.

\noindent $3)\to 4)$.  The implication is straightforward since $\varphi_{F_i(\bar{x})}\not\sim_l \varphi_{x_i}$
follows from $W_{F_i(\bar{x})}\not\sim_l W_{x_i}$

  \noindent $4)\to 2)$.  The following construction is uniform for both $l$.
  First for $i\leq 2$ we define  functions  $G_i$ and $T_i$ which are computable with the oracle $z_l$.
  \begin{align*}
  &G_i(\bar{x},\bar{y})=b_i(\bar{y})\mbox{  if only } h(\bar{y})\in W^{z_l}_{x_i}\\
  & T_i(\bar{x},\bar{y})=a_i(\bar{y})\mbox{  if only } h(\bar{y})\in A_i.
  \end{align*}
  By Reduction principle for function graphs \cite{Rogers} we find a    function
  $E_i(\bar{x},\bar{y})$ with the following properties:
  \begin{itemize}
  \item  $E_i(\bar{x},\bar{y})$ is computable with the oracle $z_l$, therefore for  some computable    function      $g_i:\omega^3\to \omega$,     $E_i(\bar{x},\bar{y})=K^{4,z_l}(g_i(\bar{x}),\bar{y})$,
      where $K^{4,z_l}$ is Kleene universal function for 3-arity functions computable with the oracle $z_l$.
  \item If $h(\bar{y})\in W^{z_l}_{x_i}\setminus A_i$ then $E_i(\bar{x},\bar{y})=b_i(\bar{y})$.
  If $h(\bar{y})\in A_i\setminus W^{z_l}_{x_i}$ then $E_i(\bar{x},\bar{y})=a_i(\bar{y})$.
  \end{itemize}
If $l=1$ by Smullyan Theorem \cite{Smullyan} there exist three computable functions
$n_0,\,n_1,\,n_2:\omega^3\to\omega$ such that
  \begin{align*}
  \varphi_{K^{4}(g_i(\bar{x}),n_0(\bar{g}(\bar{x}),n_1(\bar{g}(\bar{x})),n_2(\bar{g}(\bar{x}))))}= \varphi_{n_i(\bar{g}(\bar{x}))},
   \end{align*}


 \noindent  If $l=2$ there exist three computable functions
$n_0,\,n_1,\,n_2:\omega^3\to\omega$ such that
  \begin{align*}
  \varphi_{K^{4,z_2}(g_i(\bar{x}),n_0(\bar{g}(\bar{x}),n_1(\bar{g}(\bar{x})),n_2(\bar{g}(\bar{x}))))}=^\ast \varphi_{n_i(\bar{g}(\bar{x}))},
   \end{align*}
   under  the condition that $K^{4,z_2}(g_i(\bar{x}),\bar{n}(\bar{g}(\bar{x})))\downarrow$
 The last statement requires more details which we show below.
 It is easy to see that, for $i\leq 2 $,  $f_i(\bar{z},\bar{y})\leftrightharpoons K^{4,z_2}(z_i,\bar{y})$ is $\Sigma^o_2$-function. Therefore by Remark~\ref{rmalmosteq} there exist computable functions $\lambda_i$ such that
 for $(\bar{z},\bar{y})\in\dom(f_i)$
  \begin{align*}
  \varphi_{f_i(\bar{z},\bar{y})}=^\ast \varphi_{\lambda_i(\bar{z},\bar{y})}.
  \end{align*}

  By  Smullyan Theorem,
    \begin{align*}
  \varphi_{\lambda_i(\bar{z},\bar{n}(\bar{z}))}= \varphi_{n_i(\bar{z})}
     \end{align*}
     Therefore
    \begin{align*}
  \varphi_{f_i(\bar{z},\bar{n}(\bar{z}))}=^\ast \varphi_{n_i(\bar{z})}
  \end{align*}
 for $(\bar{z},\bar{n}(\bar{z}))\in\dom(f_i)$.
 Now we are ready do define $H:\omega^3\to\omega$:
  \begin{align*}
  H(\bar{x})=h(n_0(\bar{g}(\bar{x})),n_1(\bar{g}(\bar{x})),n_2(\bar{g}(\bar{x}))).
   \end{align*}
   Let us sow that $H(\bar{x})=h(n_0(\bar{g}(\bar{x})),n_1(\bar{g}(\bar{x})),n_2(\bar{g}(\bar{x}) ))$ is a required computable function. Assume contrary that for some $x_0,\, x_1$ and $x_2$

 \[
H(x_0,x_1,x_2)\not\in
\Big( (A_i\cap W^{z_l}_{x_i})\cup (\overline{A_i}\cap \overline{W}^{z_l}_{x_i})
\Big)
\]
for some $i \in \{0, 1, 2 \} $.
  We have two cases:

  \begin{enumerate}
  \item[(a)] $H(\bar{x})\in A_i\setminus W^{z_l}_{x_i}$
  \item[(b)] $H(\bar{x})\in  W^{z_l}_{x_i}\setminus A_i$.
  \end{enumerate}
 It is worth noting that in both cases $K^{4,z_l}(g_i(\bar{x}),\bar{n}(\bar{g}(\bar{x})))\downarrow$.
In the case $(a)$,
 \begin{align*}
& \varphi_{F_i(\bar{n}(\bar{g}(\bar{x})))}=\varphi_{a_i(\bar{n}(\bar{g}(\bar{x})))} \mbox{ by the definition of $F_i$}\\
& \varphi_{a_i(\bar{n}(\bar{g}(\bar{x})))}= \varphi_{K^{4,z_l}(g_i(\bar{x}),\bar{n}(\bar{g}(\bar{x})))} \mbox{ by the definition of $g_i$}\\
& \varphi_{K^{4,z_l}(g_i(\bar{x}),\bar{n}(\bar{g}(\bar{x})))}= \varphi_{n_i(\bar{g}(\bar{x}))}\mbox{ by the choise of $\bar{n}$}
\end{align*}

This contradicts the condition on $F_i$.
In the case $(b)$,
 \begin{align*}
& \varphi_{F_i(\bar{n}(\bar{g}(\bar{x})))}=\varphi_{b_i(\bar{n}(\bar{g}(\bar{x})))} \mbox{ by the definition of $F_i$}\\
& \varphi_{b_i(\bar{n}(\bar{g}(\bar{x})))}=\varphi_{K^{4,z_l}(g_i(\bar{x}),\bar{n}(\bar{g}(\bar{x})))} \mbox{ by the definition of $g_i$}\\
&\varphi_{K^{4,z_l}(g_i(\bar{x}),\bar{n}(\bar{g}(\bar{x})))}= \varphi_{n_i(\bar{g}(\bar{x}))}\mbox{ by the choise of $\bar{n}$}
\end{align*}

This contradicts the condition on $F_i$.
Therefore $H$ is a required productive function.
\end{proof}
\section{  Index sets that we need}\label{index_sets}

Let us fix the following index sets:
\begin{align*}
&E_0=\{n\mid \pi_n\subseteq \{0,1,2\}\wedge\big((W_n \mbox{ is  finite})\vee \varphi_n\cap  \omega\times\{0\}=^\ast\emptyset\big)\};\\
&E_1=\{n\mid \pi_n\subseteq \{0,1,2\}\wedge\big((W_n \mbox{ is  finite})\vee \varphi_n\cap  \omega\times\{1\}=^\ast\emptyset\big)\};\\
&E_2=\{n\mid \pi_n\subseteq \{0,1,2\}\wedge\big((W_n \mbox{ is  finite})\vee \varphi_n\cap  \omega\times\{2\}=^\ast\emptyset\big)\}.
\end{align*}
Using Proposition~\ref{criteria_m_compl} we show that  the 3-tuple $(E_1,E_1,E_2)$  is m-complete in the class of  $3$  -tuples of $\Sigma^0_2$-set.

In order to define $(F_0,F_1,F_2)$ we first take  a computable function $h:\omega^3\to\omega$ such that

 \begin{eqnarray*}
\varphi_{h(x_0,x_1,x_2)}(3k)=\left \{
\begin{array}{ll}
0 & \mbox{ if }  \varphi_{x_0}(k)\downarrow\\
\uparrow &\mbox{ otherwise},
\end{array}
\right.
\end{eqnarray*}

 \begin{eqnarray*}
\varphi_{h(x_0,x_1,x_2)}(3k+1)=\left \{
\begin{array}{ll}
1 & \mbox{ if }  \varphi_{x_1}(k)\downarrow\\
\uparrow &\mbox{ otherwise},
\end{array}
\right.
\end{eqnarray*}

 \begin{eqnarray*}
\varphi_{h(x_0,x_1,x_2)}(3k+2)=\left \{
\begin{array}{ll}
2 & \mbox{ if }  \varphi_{x_2}(k)\downarrow\\
\uparrow &\mbox{ otherwise}.
\end{array}
\right.
\end{eqnarray*}
 Now we pick up appropriate functions $a_i$ and $b_i$ for $i\leq 2$.
The functions $b_i$ is defined by $\varphi_{b_i(\bar{x})}(k)=\varphi_{x_i}(k)+1$, $a_i$ is an index of $\rm{id}$ function.
Then we define  for $i\leq 2$
  \begin{eqnarray*}
F_i(\bar{x})=\left \{
\begin{array}{lll}
a_i & \mbox{ if }&  h(\bar{x})\in E_i\\
b_i(\bar{x}) &\mbox{ if }&  h(\bar{x})\not\in E_i.
\end{array}
\right.
\end{eqnarray*}
Let us show that $\varphi_{F_i(\bar{x})}\not\sim_2 \varphi_{x_i}$ for $i\leq 2$.
 Without loss of generality it is sufficient to consider $i=0$.
Assume $h(\bar{x})\not\in E_0$. Then $W_{h(\bar{x})}$ is infinite and infinitely often a value of $\varphi_{h(\bar{x})}$ is  zero.
By   construction, $W_{x_0}$ is infinite. Since for $k\in\omega$,
\begin{align*}
\varphi_{F_0(\bar{x})}(k)=\varphi_{b_0(\bar{x})}(k)=\varphi_{x_0}(k)+1,
\end{align*}
 there exist infinitely many $k\in\omega$ such that $\varphi_{F_0(\bar{x})}(k)\neq\varphi_{x_0}(k)$. As a corollary,
$\varphi_{F_0(\bar{x})}\neq^\ast\varphi_{x_0}$.
Assume $h(\bar{x})\in E_0$. Then  $\varphi_{x_0}$ is a finite function. Since for $k\in\omega$,
\begin{align*}
\varphi_{F_0(\bar{x})}(k)=\varphi_{a_0(\bar{x})}(k)= k,
\end{align*}
 we have $\varphi_{F_0(\bar{x})}\neq^\ast\varphi_{x_0}$.
\begin{lemma}

Let a $3$-tuple $(E_1,E_1,E_2)$ be m-complete in the class of  $3$-tuples of $\Sigma^0_2$-set.
Then the following 3-tuple
\begin{align*}
& Y_0=E _1\cap E_1\cap E_2\\
& Y_1= E_1\cap E_2\\
&Y_2=E_2\\
\end{align*}
is  m-complete in the class of  $3$-tuples  $(X_0,X_1,X_2)$ of $\Sigma^0_2$-set with the additional condition
$X_0\subseteq X_1\subseteq X_2$.
\end{lemma}

 \noindent By the choice of $(E_0,E_1,E_2)$  at the beginning of this section we have $Y_0\subseteq Y_1\subseteq Y_2$, where
\begin{align*}
&Y_0=\{n\mid \pi_n\subseteq \{0,1,2\}\wedge W_n \mbox{ is  finite}\};\\
&Y_1=\{n\mid \pi_n\subseteq \{0,1,2\}\wedge\big(W_n \mbox{ is  finite}\vee
(\exists N)(\forall k\geq N)(\varphi_n(k)\downarrow\,\rightarrow\, \varphi_n(k)=0
\};\\
&Y_2=\{n\mid \pi_n\subseteq \{0,1,2\}\wedge\big(W_n \mbox{ is  finite}\vee \varphi_n\cap  \omega\times\{2\}=^\ast\emptyset\big)\}.
\end{align*}

\begin{remark}
It is worth noting that if 3-tuple $(Z_0,Z_1,Z_2)$ is $m$-complete in  the class of 3-tuple  $(X_0,X_1,X_2)$
of $\Sigma^0_2$-sets with the additional condition $X_0\subseteq X_1\subseteq X_2$ then the set
$Z_0\cup(Z_2\setminus Z_1)\not\in \Sigma^0_2$ since this combination is the $m$-greatest among the similar combinations of $\Sigma^0_2$-sets.
\end{remark}

\section{Classes $K$    without computably presentable $\KR$}\label{without_constr}

Below we list important restrictions on a class $K$ of total computable numerical functions:

\begin{enumerate}
\item Among the basic functions $\BF$ the class contains $+,\dotminus,\cdot, [\frac{x}{2}]$ and
$[\sqrt x]$.

\item  The class has a computable universal function  for all unary functions i.e. the sequence $\{\mathcal{F}_n\}_{n\in\omega}$ of all unary functions from $ K$ is computable.
\item There exists a computable function $H:\omega\to\omega$ such that for all $n\in\omega$
$\im(\mathcal{F}_{H(n)})= \{0\}\cup \{x+1\mid x\in W_n \}$. Moreover, for all $i\in W_n$,
$\mathcal{F}^{-1}_{H(n)}(i+1)$ is an infinite set.
\item The class is closed under superposition and either under the standard bounded recursion scheme or the following   recursion scheme 2:
If $\alpha\,\,g,\psi\in K$ and $f$ is defined by

  \begin{eqnarray*}
f(\bar{x},y)=\left \{
\begin{array}{lll}
\alpha(\bar{x})&\mbox{ if }& y=0\\
\psi(\bar{x},y,f(\bar{x},[\frac{y}{2}])) &\mbox{ if }& y\geq 1
\end{array}
\right.
\end{eqnarray*}
 and

\begin{align*}
f(\bar{x},y)\leq g(\bar{x},y)
\end{align*}
then $f\in K$.
\end{enumerate}

\begin{remark}\label{r_2}
It is worth noting that
 the class $K$ satisfying the restrictions above  contains all almost  constant functions,
  Cantor 3-tuple $(c,l,r)$, $\rm{sg}|x-y|$.

 \end{remark}
\begin{theorem}Let $K$  satisfy the requirements above and $\KR$ be a structure generated by $\K$.
 Then $\KR$ does not have a computable presentation.
\end{theorem}
 The proof follows from Theorem~\ref{sigma02construct}   and the claim $A_0\cup(A_2\setminus A_1)\not\in \Sigma^0_2$ which  is based on the following proposition and Section~\ref{index_sets}.

\begin{proposition}
 $(Y_0,Y_1,Y_2)\leq_m (A_0,A_1,A_2)$.

\end{proposition}
 \begin{proof}
 We are going to construct a computable function $f$ such that $n\in Y_i\leftrightarrow f(n)\in A_i$. In order to do that we will construct a
  computable sequence $\{F_n\}_{n \in \omega}$ of computable functions
  by steps and then effectively find a  required reduction  $f$.

We take
\begin{itemize}
\item  a standard computable reduction function $\alpha:\omega\to\omega$ for $\rm{Fin}\leq_m \omega\setminus\rm{Tot}$ (c.f.  \cite{Shoenfield}) with the following  properties:
    \begin{itemize}
 \item if  $W_n$ is finite then $W_{\alpha(n)}$ is finite;
 \item if  $W_n$ is infinite then $\varphi_{\alpha(n)}$ is total;
\item  $\pi_n=\pi_{\alpha(n)}$;
\item If $(\exists^\infty a)\,\varphi_n(a)=x$ then $(\exists^\infty b)\,\varphi_{\alpha(n)}(b)=x$ and vice versa.
 \end{itemize}
\item a  computable function $t$ such that
\begin{itemize}
\item $W_{t(n)}= \{<k,d>\mid \varphi_{\alpha(n)}(k)=d \}$.
\end{itemize}
\end{itemize}



Now we point out the requirements on a step $s$  which we want to meet in our construction:
\begin{itemize}
\item $F^{s+1}_n \supseteq F^{s}_n$
\item $\dom(F^s_n)=[0,\dots,m^s_n]$ is a proper  initial segment of $\omega$.
\item
 If $\mathcal{F}_{H(t_s(n))}(s)=<c,d>+1\wedge d>2$ then
 $(\exists j)\, F_n(j)>2$.
  \item If  $\mathcal{F}_{H(t_s(n))}(s+1)=<i,2>+1$
 then in the process of the construction we provide the following:
 $F_n\not\in\{\mathcal{F}_0,\dots,\mathcal{F}_i\}$.
\item $F_n=\bigcup_{s\in\omega} F^s_n$.
  \end{itemize}

\noindent \noindent  W.l.o.g. we assume now that K is closed under the recursion scheme 2 since the case
when K is closed under the primitive recursion is much more easy so it is left to a reader.

  \noindent {\bf Description of a construction of $m^s_n$, $\{F^s_n\}_{n,s\in\omega}$, $t_s(n)$ and $I^{s}_n$}:

\noindent{\bf Step 0}

$m^n_0=[0]$, $F^0_n(0)=0$, $t_0(n)=t(n)$ and $I^0_n=0$.

\noindent{\bf Step s+1}

\noindent {\bf Case 1} If  $\mathcal{F}_{H(t_s(n))}(s+1)=<i,0>+1$  or $\mathcal{F}_{H(t_s(n))}(s+1)=0$  then we proceed as follows:

  \begin{eqnarray*}
m^{s+1}_n=\left \{
\begin{array}{lll}
m^s_n &\mbox{ if }& m^s_n>0\\
1 &\mbox{ if }& m^s_n=0
\end{array}
\right.
\end{eqnarray*}
and for all $j \leq m^s_n$, $F^{s+1}_n (j) = F^s_n (j)$,  for all $m^s_n \leq j\leq m^{s+1}_n$ we put  $F^{s+1}_n(j)= F^{s}_n(m^s_n)$, $t_{s+1}(n)=t_s(n)$ and $I^{s+1}_n=I^{s}_n$.

\noindent {\bf Case 2}  $\mathcal{F}_{H(t_s(n))}(s+1)=<i,1>+1$ we consider the following subcases:

\begin{enumerate}
\item[Subcase 2.1] $i> s+1$.   We proceed as in Case~1.
\item[Subcase 2.2]$i\leq s+1$ and  $i\leq I^{s}_n$.   We proceed as in Case~1.
\item[Subcase 2.3] $ I^{s}_n < i\leq s+1$
 We proceed as follows:
  \begin{eqnarray*}
m^{s+1}_n=\left \{
\begin{array}{lll}
m^s_n &\mbox{ if }& m^s_n>0\\
1 &\mbox{ if }& m^s_n=0
\end{array}
\right.
\end{eqnarray*}
and for all $j \leq m^s_n$, $F^{s+1}_n (j) = F^s_n (j)$, for all $m^s_n \leq j< m^{s+1}_n$ we put  $F^{s+1}_n(j)= F^{s}_n(m^s_n)$ and for $F^{s+1}_n(m^{s+1}_n)$ we chose the least value
from $\{0,1\}$ such that $F^{s+1}_n(m^{s+1}_n)\neq F^{s}_n(m^s_n)$, $t_{s+1}(n)=t_s(n)$ and $I^{s+1}_n=i$.
\end{enumerate}

\noindent {\bf Case 3}
 If  $\mathcal{F}_{H(t_s(n))}(s+1)=<i,2>+1$  then for all $j \leq m^s_n$, $F^{s+1}_n (j) = F^s_n (j)$ and  we proceed as follows:
$m^{s+1}_n=m^s_n+2(i+1)$ and for $k\leq i$   we chose the  values  of $F^{s+1}_n$ on the arguments $m^n_{s}+2k+1$ and $m^n_{s}+2k+2$ according with the following table:
\medskip

\begin{tabular}{ |p{2.5cm}|p{2.5cm}|p{2.7cm}| p{2.7cm}| }

%
$\mathcal{F}_k(m^s_n+2k+1)$& $\mathcal{F}_k(m^s_n+2k+2)$& $F^{s+1}_n(m^n_{s}+2k+1)$ & $F^{s+1}_n(m^n_{s}+2k+2)$\\
\hline
2& 2 &0&0\\
\hline
2& 0 &0&0\\
\hline
2& 1 &0&0\\
\hline
1& 2 &0&0\\
\hline
1& 1 &0&0\\
\hline
1& 0 &2&2\\
\hline
0& 2 &2&2\\
\hline
0& 1 &2&2\\
\hline
0& 0 &2&2\\
\hline
\end{tabular}
\medskip

\noindent The equality
$W_{t_{s+1} (n) } = W_{t_s(n)} \setminus \{<i,2>\}$
defines the value of $t_{s+1}(n)$ and  $I^{s+1}_n = I^s_n$.

\noindent {\bf Case 4}
 If  $\mathcal{F}_{H(t_s(n))}(s+1)=<i,d>+1$ and $d>2$  then  for all $j \leq m^s_n$, $F^{s+1}_n (j) = F^s_n (j)$ and  we proceed as  follows:
$ m^{s+1}_n = m^s_n +1$, $F^{s+1}_n (  m^{s+1}_n  ) = d$,  $t_{s+1}(n)=t_s(n)$
 and  $I^{s+1}_n = I^s_n$.


 We put $F_n=\bigcup_{s\in\omega} F^s_n$ and effectively find  computable function $f:\omega\to\omega$ such that
 %
 $\varphi_{f (n)}  = (F_n \cap (W_ {\alpha (n)}  \times \omega ) ) \cup \{ <x,d> | F_n(x)=d \wedge d>2 \}$.

  Now we show that $f$ is a required reduction.


If $n\in Y_0$ then  $W_{\alpha(n)}$ is finite, so is $\varphi_{f(n)}$  and $f(n)\in A_0$.
If $n\not\in Y_0$ then there are two cases:

\noindent $1)$  $W_{\alpha(n)}$ is $\omega$, by construction, $\varphi_{f(n)}=F_n$ and $F_n$ is total, so  $f(n)\not\in A_0$.

\noindent $2)$ $\pi_n \not\subseteq \{0,1,2  \}$. Then for some $ j\in\omega$  $F_n (j) > 2$, so $ \pi_{f(n)} \not\subseteq \{0,1,2  \}$. Again  $f(n) \not\in A_0$.
 {\bf So we have $f^{-1}(A_0)=Y_0$}.


If $n\in Y_1\setminus Y_0$ then $\varphi_{\alpha(n)}$ is total and $\varphi_{\alpha(n)}=^* 0$. In this case $\varphi_{f(n)}=F_n$ and by construction $F_n\in \AC$ since $(\exists s_1\in \omega)(\forall s\geq s_1)\, I^n_s=I^n_{s_1}$. So $\varphi_{f(n)}=F_n$ and $f(n)\in A_1\setminus A_0$.


If $n\not\in Y_2$ then $\varphi_{\alpha(n)}$ is total and $\varphi_{f(n)}=F_n$. By construction the case $\mathcal{F}_{H(t_s(n))}(s+1)=<i,2>+1$ arises infinitely often and the collection of the corresponding numbers $i$ is infinite too.
 So, for infinitely many $i$,
either
${\rm im}(F_n)\not\subseteq \{0,1,2\}$ or
$\overline{F_n}\not\in \{\overline{\mathcal{F}_0},\dots,\overline{\mathcal{F}_i}\}$. Hence either ${\rm im}(F_n)\not\subseteq \{0,1,2\}$ or  $\overline{F_n}\not\in \KR$. That means $f(n)\not\in A_2$.

%
If $n\in Y_2\setminus Y_1$ then $\varphi_{\alpha(n)}=F_n$ is total. By the choice of $n$,
$(\exists N)(\forall i\geq N)\, \varphi_{\alpha(n)}(i)\neq 2$. Hence $(\exists^\infty i)\varphi_{\alpha(n)}(i)= 1$.

Let us note that after some step $s_0$ for all $i$ we have $<i,2>\not\in W_{t_s(n)}$ for $s\geq s_0$. We define
$t_\infty(n)=t_{s_0}(n)$. It is easy to see that, for $s\geq s_0$, $t_\infty(n)=t_s(n)$.
On the step $s+1$, when $\mathcal{F}_{H(t_\infty(n))}(s+1)=<i,1>+1$,  $F^{s+1}_n(m^{s+1}_n)\neq F^{s+1}_n(m^{s+1}_n-1)$.
Hence $F_n\not\in \AC$, so $f(n)\not\in A_1$. {\bf So we have $f^{-1}(A_1)=Y_1$}.

Using $s_0$ and $N$ from above we explain that $F_n\in \K$.

 Let $m_n=m^{s_0+1}_{n}$.
 It is easy to see
 that the following functions  belongs to $K$:
 \begin{itemize}
 \item the characteristic function of the set $A=\{m_n\cdot 2^i\mid i\geq 0\}$;
\item  the function $g(x)$, that computes $\max\{y\in A\mid y\leq x\}$ for $x\geq m$ and for $x< m$ it is equal to $0$;
\item  the function $S(x)=\mu(s')(F^{s'}_n(x)\downarrow)$.
 \end{itemize}
 In order to meet our goal we construct the function $I(x)=I^{S(x)-1}_n$ by the following rules:

 Assume $I_0=I^{s_0}_n$. Then we define
  \begin{itemize}
  \item for $x<m_n$, $I(x)=0$,
  \item for $x=m_n$, $I(x)= I_0$,
  \item for $x>m_n$ and $g(x)>m_n$,
\small{
   \begin{eqnarray*}
I(x)=\left \{
\begin{array}{rl}
&l(\mathcal{F}_{H(t_\infty(n))}(S(x)-1)-1)\leftrightharpoons i \mbox{ if }\\
 &I([\frac{x}{2}])\leq i\leq S(x)-1\wedge r(\mathcal{F}_{H(t_\infty(n))}(S(x)-1)-1)=1\\
&I([\frac{x}{2}]) \mbox{ otherwise . }
\end{array}
\right.
\end{eqnarray*}
}
  \end{itemize}
   \item for $x>m_n$ and $g(x)=m_n$, $I(x)= I_0$.

   From above we can see that $\lambda x.I(x)\in \K$.


We can assume that $x\geq N$ and $S(x)>s_0$.

Suppose $x\in A$.

 If $\mathcal{F}_{H(t_\infty(n))}(S(x))=<i,0>+1$ then $F_n(x)=F_n([\frac{x}{2}])$.
 The same is done if  $\mathcal{F}_{H(t_\infty(n))}(S(x))=<i,1>+1$ but $i\leq I(x)$ or $i>S(x)$.

 Otherwise, i.e., if  $\mathcal{F}_{H(t_\infty(n))}(S(x))=<i,1>+1$ and $I(x)\leq i\leq S(x)$ then for the value of $F_n(x)$ we chose the first one from $\{0,1\}$ which differs from $F_n([\frac{x}{2}])$.

Suppose $x\not\not\in A$.

 If $\mathcal{F}_{H(t_\infty(n))}(S([\frac{x}{2}]+1)=<i,0>+1 \mbox{or } 0$ then $F_n(x)=F_n([\frac{x}{2}])$.
 The same is done if $\mathcal{F}_{H(t_\infty(n))}(S([\frac{x}{2}]+1)=<i,0>+1$ but $i\leq I([\frac{x}{2}]+1)$ or
  $i>S([\frac{x}{2}]+1)$.
  Otherwise, i.e., if  $\mathcal{F}_{H(t_\infty(n))}(S([\frac{x}{2}]))=<i,1>+1$ and $I([\frac{x}{2}])\leq i\leq S([\frac{x}{2}])$ %
  then for the value of $F_n(x)$ we chose the first one from $\{0,1\}$ which differs from $F_n([\frac{x}{2}])$.
 Therefore the scheme above shows that $F_n\in\K$. {\bf So we have $f^{-1}(A_2)=Y_2$}.

\end{proof}

\begin{corollary}
The structures of $\Pt$-numbers, $\SPt$-numbers and $\mathcal{E}_n$-numbers, $n\geq 2$ do not have computable copies.
\end{corollary}

\begin{corollary}
The fields of $\Pt$-numbers and $\mathcal{E}_n$-numbers, $n\geq 3$ do not have computable copies.
\end{corollary}
The claim follows from real closedness of the corresponding fields.


\bigskip
\begin{thebibliography}{1}
%
\bibitem{ALR}
Allender E., Loui M.C. and Regan K.W.,
Complexity classes, In handbook: Algorithms and theory of computation handbook: general concepts and techniques,
Chapman and Hall CRC,  2010.




\bibitem{Arslanov_0}
Arslanov,~M.~M.,
A local Theory of Degrees of Unsolvability and $\Delta^0_2$-sets, Kazan University Press, 1987.  (in Russian)

\bibitem{Arslanov}
Arslanov,~M.~M.,
On some generalisations of  a fixed point theorem.  \emph{news of universities, Mathematics} No 5, 1981, pp 9--16.
 (in Russian)
%
\bibitem{Ersh_TH}
 Ershov,~Yu.~L.,
 Theory of numberings.  \emph{ in: E.R. Griffor, ed., Handbook of Computability
Theory. Elsevier Science B.V., Amsterdam}, 1999, pp.  473--503.
%
%
  \bibitem{ErshGonch}
Ershov,~Yu.~L. and Goncharov~S.~.S., Constructive models, Siberian School of Algebra
and Logic, Consultants Bureau, New York, 2000.
%
  \bibitem{Gonch}
Goncharov~S.~S.,
Countable Boolean algebras and decidability, Siberian School of Algebra
and Logic, Consultants Bureau, New York, 1997.
%
  \bibitem{Gonch1}
 Goncharov~S.~S.,
Countable Boolean Algebras and Decidability, Siberian School of Algebra and Logic, Springer, 1997.
%
  \bibitem{Grzegorczyk}
Grzegorczyk~A., Some classes of recursive functions, Rozprawy Mathemaczne IV, Warszawa, 1953.
%
  \bibitem{Khisamiev}
 Khisamiev~N., Constructive abelian groups, Handbook of recursive mathematics,
Vol. 2, Stud. Logic Found. Math., vol. 139, North-Holland, Amsterdam, 1998, pp. 1177--1231.
%
\bibitem{Ko}
Ko, Ker-I, Complexity Theory of Real Functions, Progress in Theoretical Computer Science, Springer, 1991.
%
\bibitem{KK_spectr}
Korovina~M. and
               Kudinov~O., Spectrum of the computable real numbers. Algebra Log. 55(6), 2017, pp. 485–500.

\bibitem{KK_up}
 Korovina M. and
                Kudinov O.,
The Uniformity Principle for Sigma-definability,
J. Log. Comput.,
  vol.19,
  No. 1, 2009,
   pp. 159--174.
%

\bibitem{Miller}
Miller~M, Gonzales~V. O.,
Degree spectra of real closed fields, Archive for Mathematical Logic  58, pp. 387–411, 2019.
%
\bibitem{Morozov}
 Morozov~A.S.,
Countable homogeneous boolean algebras, Algebra and Logic, 21, No. 3, 1982, oo. 269--282.
%
 \bibitem{Rogers}
Rogers~ H.,
  \emph{ Theory of Recursive Functions and Effective Computability}.
 McGraw-Hill, New York, 1967.
 %
\bibitem{Shoenfield}
Shoenfield~J.~R.,
 \emph{  Degrees of unsolvability},
 North-Holland Publ., 1971.

  \bibitem{Smullyan}
 Smullyan~R.M., Theory of Formal Systems, Annals of mathematics studies, 47, Princeton, N.J.
  \bibitem{Soarehbook}
 Soare~R.I.,
 \emph{ Recursively Enumerable Sets and Degrees: A Study of Computable Functions and Computably Generated Sets}.
 Springer Science and Business Media, 1997.
 %
  \bibitem{Khisamiev}
 Khisamiev~N.,
  Constructive abelian groups, Handbook of recursive mathematics,
Vol. 2, Stud. Logic Found. Math., vol. 139, North-Holland, Amsterdam, 1998, pp. 1177--1231.
%

%
 \bibitem{Weihrauchbook}
 Weihrauch~K.,
 \emph{ Computable Analysis} 2000.
 Springer Verlag.
 %
 %

%
\end{thebibliography}
\end{document}